\documentclass{CSML}

\pdfoutput=1

\usepackage{hyperref}
\hypersetup{hidelinks}

\usepackage{lastpage}
\newcommand{\dOi}{13(3:14)2017}
\lmcsheading{}{1--\pageref{LastPage}}{}{}%
{Jan.~27, 2017}{Aug.~23, 2017}{}

\usepackage{amssymb,amscd,amsthm,latexsym}

\usepackage[all]{xy}
 \usepackage{color}
\theoremstyle{plain}
\newtheorem{theorem}[thm]{Theorem}
\newtheorem{lemma}[thm]{Lemma}
\newtheorem{propo}[thm]{Proposition}
\newtheorem{ex}[thm]{Example}

\theoremstyle{definition}
\newtheorem{remark}[thm]{Remark}

\newtheoremstyle{defC}%
  {6pt}
  {6pt}
  {\normalfont}
  {}
  {\bfseries}
  {{\bfseries .}}
  {5pt plus 1pt minus 1pt}
  {\thmname{#1} \thmnumber{#2} \thmnote{\normalfont#3}}
\theoremstyle{defC}
\newtheorem{defiC}[thm]{Definition}

\newtheoremstyle{thmC}%
  {6pt}
  {6pt}
  {\itshape}
  {}
  {\bfseries}
  {{\bfseries .}}
  {5pt plus 1pt minus 1pt}
  {\thmname{#1} \thmnumber{#2} \thmnote{\normalfont#3}}

\theoremstyle{thmC}
\newtheorem{thmC}[thm]{Theorem}
\newtheorem{propC}[thm]{Proposition}
\newtheorem{lemmaC}[thm]{Lemma}

\newdir{ >}{
  @{}*!/-10pt/@{>} }

\newcommand{\CC}{ \ensuremath{\mathbb {C}} }
\newcommand{\Equiv}{\ensuremath{\mathrm{Equiv}}}
\newcommand{\Conn}{\ensuremath{\mathrm{Conn}}}
\newcommand{\Grpd}{\ensuremath{\mathrm{Grpd}}}
\newcommand{\Cat}{\ensuremath{\mathrm{Cat}}}
\newcommand{\RG}{\ensuremath{\mathrm{RG}}}

\newcommand{\Eq}{ \ensuremath{\mathrm{Eq}} }

\newcommand{\map}[2]{ \ensuremath{ \xymatrix@1@C=15pt{ #1 \ar[r] & #2 } } }
\newcommand{\mono}[2]{ \ensuremath{ \xymatrix@1@C=15pt{ #1 \ar@{ >->}[r] & #2 } } }
\newcommand{\regepi}[2]{ \ensuremath{ \xymatrix@1@C=15pt{ #1 \ar@{>>}[r] & #2 } } }

\begin{document}
\title[Some remarks on connectors and groupoids in Goursat categories]%
{Some remarks on connectors and groupoids\\ in Goursat categories}

\dedicatory{Dedicated to Ji\v r\'{\i}  Ad\' amek on the occasion of his seventieth birthday}

\author[Marino Gran]{Marino Gran\rsuper{a}}
\address{\lsuper{a}Institut de Recherche en Math\'ematique et Physique, Universit\'e Catholique de Louvain,
Belgium}

\email{\{marino.gran,idriss.tchoffo\}@uclouvain.be}

\author[Idriss Tchoffo Nguefeu]{Idriss Tchoffo Nguefeu\rsuper{a}}


\author[Diana Rodelo]{Diana Rodelo\rsuper{b}}
\address{\lsuper{b}Departamento de Matem\'atica, Universidade
do Algarve, 3001-501 Coimbra, Portugal}

\thanks{The third author acknowledges partial financial assistance by Centro de Matem\'{a}tica da
Universidade de Coimbra---UID/MAT/00324/2013, funded by the
Portuguese Government through FCT/MCTES and co-funded by the
European Regional Development Fund through the Partnership
Agreement PT2020.}
\email{drodelo@ualg.pt}



\begin {abstract}
We prove that connectors are stable under quotients in any (regular) Goursat category. As a consequence, the category $\Conn(\mathbb{C})$ of connectors in $\mathbb{C}$ is a Goursat category whenever $\mathbb C$ is. This implies that Goursat categories can be characterised in terms of a simple property of internal groupoids.
\end {abstract}



\maketitle


Over the last twenty years the property of $n$-permutability of congruences in a variety of universal algebras has been investigated from a categorical perspective (see \cite{ckp,mp,JRVdL}, for instance, and references therein). When $\mathbb{C}$ is a regular category, the $2$-permutability property, usually called the \emph{Mal'tsev} property, is a concept giving rise to a beautiful theory, whose main features are collected in \cite{bb}. Many important results still hold when a regular category $\mathbb{C}$ satisfies the strictly weaker property of $3$-permutability, namely the \emph{Goursat} property. A nice feature of a (regular) Goursat category $\mathbb{C}$ is that the lattice of equivalence relations on any object in $\mathbb{C}$ is a modular lattice \cite{ckp}, a property that plays a crucial role in commutator theory~\cite{FM, hg}.

The aim of this paper is twofold: first of all we establish some basic properties of Goursat categories in terms of connectors \cite{bmaconn}, as it was done in \cite{bmaconn} for the case of Mal'tsev categories. These results have turned out to be useful to develop a monoidal approach to internal structures~\cite{ght}. We then give a new characterisation of Goursat categories in terms of properties of (internal) groupoids, on the model of what was done in \cite{ma} in the case of Mal'tsev categories.

In the first section, we recall the main properties of Goursat categories that
will be needed throughout the paper. In Section $2$ we prove that for any
Goursat category $\mathbb{C}$, the category $\Equiv(\mathbb{C})$ of equivalence
relations in $\mathbb{C}$ is also a Goursat category (Proposition
\ref{equiGour}, see also \cite{bourn}). We use this result to give some properties of Goursat categories in terms of connectors in Section $3$. More precisely, we show that, when $\mathbb{C}$ is a Goursat category, then connectors are stable under quotients in $\mathbb{C}$ (Proposition \ref{stabilityepi}), and this implies that the category $\Conn(\CC)$ of connectors in $\CC$ is again a Goursat category (Theorem \ref{conngour}).

We conclude the paper by giving a new characterisation of Goursat categories in terms of properties of groupoids and internal categories (Theorem \ref{caract}). It turns out that a regular category $\mathbb{C}$ is a Goursat category if and only if the category $\Grpd(\mathbb{C})$ of groupoids (equivalently, the category $\Cat(\mathbb{C})$ of internal categories) in $\mathbb{C}$ is closed under quotients in the category $\RG(\mathbb{C})$ of reflexive graphs in $\CC$.


\section{Preliminaries}
In this section we recall some basic definitions and properties of (regular) Goursat categories, needed throughout the article. We shall always assume that the category  $\mathbb{C}$ in which we are working is a \textbf{regular category}: this means that $\CC$ is finitely complete, regular epimorphisms are stable under pullbacks, and kernel pairs have coequalisers. Equivalently, any arrow $f: A\longrightarrow B$ has a unique factorisation $f=i\circ r $ (up to isomorphism), where $r$ is a regular epimorphism and $i$ is a monomorphism and this factorisation is pullback stable; the subobject corresponding to $i$ is called the \textbf{image} of $f$.

A \textbf{relation} $R$ from  $X$ to  $Y$  is a subobject $\langle r_1,r_2 \rangle : R \rightarrowtail X \times Y $. The opposite relation of $R$, denoted $R^o$, is the relation from $Y$ to $X$ given by the subobject $\langle r_2,r_1 \rangle : R \rightarrowtail Y \times X $. A relation $R$ from $X$ to $X$ is called a relation on $X$. We shall identify a morphism $f: X \longrightarrow Y$ with the relation $\langle 1_X,f \rangle: X \rightarrowtail X \times Y$ and write $f^o$ for its opposite relation. Given another relation $\langle s_1, s_2 \rangle : S \rightarrowtail Y \times Z $ from $Y$ to $Z$, one can define the composite relation $SR$  of $R$ and $S$ as the image of the arrow $(r_1 \circ p_1, s_2 \circ p_2): R \times_Y S \longrightarrow X \times Z $, where $(R \times_Y S, p_1, p_2)$ is the pullback of $r_2: R\longrightarrow Y$ along $s_1: S \longrightarrow Y$. With the above notations, any relation $\langle r_1, r_2 \rangle: R \rightarrowtail X \times Y$ can be seen as the relational composite $r_2r_1^o$.

The following properties are well known and easy to prove. We collect them in the following lemma:

\begin{lemma}
\label{lem}
Let $f: X \longrightarrow Y$ be an arrow in a regular category  $\mathbb{C}$,
and let  $ i \circ r $ be its (regular epimorphism, monomorphism) factorisation. Then:
\begin{enumerate}
  \item $f^of$ is the kernel pair of $f$, thus $1_X \leqslant f^of $; moreover, $1_X = f^of$ if and only if $f$ is a monomorphism;
  \item $ff^o$ is $(i,i)$, thus $ff^o \leqslant 1_Y$; moreover, $ff^o = 1_Y$ if and only if $f$ is a  regular epimorphism;
  \item $ff^of = f $ and $f^off^o = f^o$.
\end{enumerate}
\end{lemma}

 \begin{defi} A relation $(R, r_1,r_2)$ on an object $X$  is said to be :
\begin{itemize}
  \item  \textbf{reflexive} when there is an arrow $ r: X \longrightarrow R $ such that $ r_1 \circ r = 1_X = r_2 \circ r $;
  \item  \textbf{symmetric} when there is an arrow  $ \sigma :  R \longrightarrow R $ such that $ r_2 = r_1 \circ \sigma $ and $r_1 = r_2 \circ \sigma $;
  \item  \textbf{transitive} when, by considering the following pullback
  $$
     \xymatrix{
     R \times_X R \ar[r]^-{p_2} \ar@{}[dr]|(0.25){\lrcorner} \ar[d]_{p_1}  & R \ar[d]^{r_1} \\ R \ar[r]_{r_2}& X,
  }
$$
there is an arrow $ t : R \times_X R \longrightarrow R$ such that $ r_1 \circ t = r_1 \circ p_1 $ and $r_2 \circ t = r_2 \circ p_2$.
 \item an \textbf{equivalence relation} if $R$ is  reflexive, symmetric and transitive.
\end{itemize}
In particular, a kernel pair $\langle f_1,f_2 \rangle: \Eq(f)\rightarrowtail X \times X$ of a morphism $f: X \longrightarrow Y$ is an equivalence relation.
The equivalence relations that occur as kernel pairs of some morphism in $\mathbb{C}$ are called {\bf effective}.
Let $\Equiv(\mathbb{C})$ be the category whose objects are equivalence relations in  $\mathbb{C}$ and arrows from $\langle r_1,r_2 \rangle : R \rightarrowtail X \times X $ to $\langle s_1,s_2 \rangle : S \rightarrowtail Y \times Y$ are pairs $(f,g)$ of arrows in $\mathbb{C}$  making the following diagram commute

\[
      \xymatrix {
   R  \ar[r]^{g} \ar@<.5ex>[d]^{r_2} \ar@<-.5ex>[d]_{r_1}  & S \ar@<.5ex>[d]^{s_2} \ar@<-.5ex>[d]_{s_1} \\ X  \ar[r]_f & Y.
}
\]
\end{defi}
When $\mathbb{C}$ is a regular category, $(R,r_1,r_2)$ is an equivalence relation on $X$ and $f: X \twoheadrightarrow Y$ a regular epimorphism, we define the \textbf{regular image of $R$ along $f$} to be the relation $f(R)$ on $Y$ induced by the (regular epimorphism, monomorphism) factorisation $\langle s_1, s_2 \rangle \circ \psi$ of the composite $(f\times f)\circ \langle r_1,r_2\rangle$:
\[
      \xymatrix{
R \ar@{.>>}[r]^{\psi} \ar@{ >->}[d]_{\langle r_1,r_2\rangle} & f(R)  \ar@{ >.>}[d]^{\langle s_1,s_2\rangle}\\
X \times X \ar@{>>}[r]_{f \times f} & Y \times Y.
  }
\]
Note that the regular image $f(R)$ can be obtained as the relational composite $f(R)=fRf^o=fr_2r_1^of^o$. When $R$ is an equivalence relation, $f(R)$ is also reflexive and symmetric. In a general regular category $f(R)$ is not necessarily an equivalence relation.
This is the case in a \emph{Goursat category} (Theorem~\ref{CKP}).

\begin{defiC}[\cite{CLP,ckp}]
A regular category $\mathbb{C}$ is called a \textbf{Goursat category} when the equivalence relations in $\mathbb{C}$ are $3$-permutable, i.e. $RSR = SRS $ for any pair of equivalence relations $R$ and $S$ on the same object.
\end{defiC}

The following characterisation will be useful in the sequel:

\begin{thmC}[\cite{ckp}]  \label{CKP} A regular category $\mathbb{C}$ is a Goursat category if and only if for any regular epimorphism $f: X \twoheadrightarrow Y$ and any equivalence relation $R$ on $X$, the regular image $f(R)= fRf^o$ of $R$ along $f$ is an equivalence relation.
\end{thmC}

There are many important algebraic examples of Goursat categories. Indeed, by a classical theorem in \cite{hm}, a variety of universal algebras is a Goursat category precisely when its theory has two ternary operations $r$ and $s$ such that the identities $r(x,y,y)= x$, $r(x,x,y)=s(x,y,y)$ and $s(x,x,y)= y$ hold. Accordingly, the categories of groups, abelian groups, modules over some fixed ring, crossed modules, quasi-groups, rings, associative algebras, Heyting algebras and implication algebras are all Goursat categories.

Any regular \emph{Mal'tsev} category is a Goursat category, thus, in particular, so is any semi-abelian category.

Many interesting properties of Goursat categories can be found in the literature (see \cite{ckp,gr,grod} and references therein).  In particular, the following characterisations will be useful for the development of this work:

\begin{thmC}[\cite{gr}]\label{Goursatpushout}
Let $\mathbb{C}$ be a regular category. The following conditions are equivalent:
\begin{enumerate}
 \item[(i)] $\mathbb{C}$ is a Goursat category;
 \item[(ii)] any commutative diagram where $\alpha$ and $\beta$ are regular epimorphisms and $f$ and $g$ are split epimorphisms in $\mathbb{C}$
  \[
    \xymatrix{
   X \ar@{->>}[r]^{\alpha} \ar@<3pt>[d]^{f}  & U \ar@<3pt>[d]^{g} \\ Y \ar@{->>}[r]_{\beta} \ar@<3pt>[u]^{s}& W \ar@<3pt>[u]^t
  }
\]
(which is necessarily a pushout) is a \textbf{Goursat pushout}: the morphism $ \lambda : \Eq(f) \longrightarrow \Eq(g)$ induced by the universal property of kernel pair $\Eq(g)$ of $g$ is a regular epimorphism.
\end{enumerate}
\end{thmC}
\noindent
We recall part of Theorem 1.3 in \cite{grod}:

\begin{thmC}[\cite{grod}]\label{newcharact}
Let $\mathbb{C}$ be a regular category. The following conditions are equivalent:
\begin{enumerate}
  \item[(i)] $\mathbb{C}$ is a Goursat category;
  \item[(ii)] for any commutative cube
  \[
      \xymatrix {
     X \times_Y Z \ar@{->>}[rr]^{\delta} \ar@<3pt>[rd] \ar@<3pt>[dd]& {}  & A \ar@<3pt>[rd] \ar@<3pt>@{-->}[dd] & {} \\
     {} & Z \ar@<3pt>[ul]  \ar@{->>}[rr]^(0.3){\gamma}  \ar@<3pt>[dd] & {} & V \ar@<3pt>[ul]  \ar@<3pt>[dd]\\
     X \ar@{-->>}[rr]^(0.7){\alpha} \ar@<3pt>[rd] \ar@<3pt>[uu]& {} & U \ar@<3pt>@{-->}[rd] \ar@<3pt>@{-->}[uu] & {} \\
     {} & Y \ar@{->>}[rr]_{\beta}  \ar@<3pt>[ul] \ar@<3pt>[uu] & {} &  W \ar@<3pt>@{-->}[ul] \ar@<3pt>[uu]
  }
\]
where the left square is a pullback of split epimorphisms, the right square is a commutative square of split epimorphisms and the horizontal arrows $\alpha$, $\beta$, $\gamma$ and $\delta$ are regular epimorphisms (commuting also with the splittings), then the right square is a pullback.
\end{enumerate}
\end{thmC}

\section{Equivalence relations in Goursat categories}
In this section we prove that $\Equiv(\mathbb{C})$ is a Goursat category for any Goursat category $\mathbb{C}$.

The category $\Equiv(\CC)$ is finitely complete whenever $\CC$ is: the terminal object in $\Equiv(\CC)$ is the discrete equivalence relation
$$ \xymatrix {
  1 \ar@<.5ex>[r]^{} \ar@<-.5ex>[r]_{} & 1
}$$
on the terminal object $1$ of $\CC$, and pullbacks are computed ``levelwise''. In particular, the kernel pair of a morphism $(f,g)$ in $\Equiv(\CC)$ is given by the kernel pairs $\Eq(f)$ of $f$ and $\Eq(g)$ of $g$ in $\CC$
\begin{equation}
\label{kernel pair in Equi(C)}
    \vcenter{\xymatrix {
  \Eq(g) \ar@<.5ex>[r]^{g_1} \ar@<-.5ex>[r]_{g_2} \ar@<.5ex>[d]^{\bar{r_2}} \ar@<-.5ex>[d]_{\bar{r_1}} & R  \ar[r]^{g} \ar@<.5ex>[d]^{r_2} \ar@<-.5ex>[d]_{r_1}  & S \ar@<.5ex>[d]^{s_2} \ar@<-.5ex>[d]_{s_1} \\ \Eq(f) \ar@<.5ex>[r]^{f_1} \ar@<-.5ex>[r]_{f_2}& X  \ar[r]_f & Y.
}}
\end{equation}
 Consequently, a morphism $(f,g)$ is a monomorphism in $\Equiv(\CC)$ if and only if $f$ and $g$ are monomorphisms in $\CC$. When $\CC$ is a Goursat category, a similar property holds with respect to regular epimorphisms:

\begin{lemma}\label{epiequiv}
Let $R$ and $S$ be two equivalence relations in a Goursat category $\mathbb{C}$ and $(f,g) : R \rightarrow S$ a morphism
\begin{equation}\label{internalfunctor}
    \vcenter{\xymatrix {
   R  \ar[r]^{g} \ar@<.5ex>[d]^{r_2} \ar@<-.5ex>[d]_{r_1}  & S \ar@<.5ex>[d]^{s_2} \ar@<-.5ex>[d]_{s_1} \\ X  \ar[r]_f & Y
}}
\end{equation}
in $\Equiv(\mathbb{C})$. Then $(f,g)$ is a regular epimorphism in $\Equiv(\mathbb{C})$ if and only if $f$ and $g$ are regular epimorphisms in $\mathbb{C}$.
\end{lemma}
\begin{proof}
When $f$ and $g$ are regular epimorphisms in $\mathbb{C}$, it is not difficult to check that $(f,g)$ is necessarily the coequaliser of its kernel pair in $\Equiv(\mathbb{C})$ given in \eqref{kernel pair in Equi(C)} (one uses the fact that $g = coeq(g_1,g_2)$ and $f = coeq(f_1,f_2)$ in $\mathbb{C}$).

Conversely, let $(f,g)$ be a morphism in $\Equiv(\mathbb{C})$ as in \eqref{internalfunctor} that is a regular epimorphism in $\Equiv(\mathbb{C})$. Consider the kernel pairs of $f$ and $g$, the (regular epimorphism, monomorphism) factorisation $f=i \circ q$ of $f$, and the regular image $(q(R), t_1, t_2)$ of $(R, r_1, r_2)$ along $q$. We obtain the following commutative diagram
\begin{equation}\label{factor}
    \vcenter{\xymatrix {
  \Eq(g) \ar@<.5ex>[r]^{g_1} \ar@<-.5ex>[r]_{g_2} \ar@<.5ex>[dd]^{\bar{r_2}} \ar@<-.5ex>[dd]_{\bar{r_1}} &  R \ar@{->>}[rd]_{\alpha} \ar[rr]^{g} \ar@<.5ex>[dd]^{r_2} \ar@<-.5ex>[dd]_{r_1}  & {} & S \ar@<.5ex>[dd]^{s_2} \ar@<-.5ex>[dd]_{s_1}  \\
  {} & {} & q(R)  \ar@<.5ex>[dd]^(0.3){t_2} \ar@<-.5ex>[dd]_(0.3){t_1} \ar@{.>}@<3pt>[ru]_(0.6){j} & {}\\
   \Eq(f)\ar@<.5ex>[r]^{f_1} \ar@<-.5ex>[r]_{f_2} &  X \ar@{->>}[rd]_q \ar[rr]^(0.8)f& {} & Y,  \\
   {} & {} & Z \ar@{ >->}[ru]_(0.6){i} & {}
}}
\end{equation}
where $(q(R), t_1, t_2) \in \Equiv(\mathbb{C})$ (by Theorem~\ref{CKP}) and $(i, j)$ is the morphism in $\Equiv(\mathbb{C})$ such that $(i, j) \circ (q, \alpha) = (f,g)$. Note that $j$ is induced from the fact that $(i\times i) \circ \langle t_1, t_2 \rangle\circ \alpha$ is the (regular epimorphism, monomorphism) factorisation of $\langle s_1, s_2 \rangle \circ g$, thus it is a monomorphism
$$
\xymatrix@C=30pt@R=25pt{
    R \ar@{>>}[r]^-{\alpha} \ar[d]_-g & q(R) \ar@{ >->}[d]^-{(i\times i)\circ \langle t_1,t_2\rangle} \ar@{ >.>}[ld]_-j \\
    S \ar@{ >->}[r]_-{\langle s_1,s_2\rangle} & Y\times Y.
}
$$
From the fact that $(f,g)$ is the coequaliser of its kernel pair in $\Equiv(\mathbb{C})$ it easily follows that $(i,j)$ is an isomorphism in $\Equiv(\mathbb{C})$.
This implies that $f$ and $g$ are regular epimorphisms in~$\mathbb{C}$.
\end{proof}

\begin{propo} \label{equiGour}~~~$\Equiv(\mathbb{C})$ is a Goursat category whenever $\CC$ is.\end{propo}

\begin{proof} As mentioned above, the category $\Equiv(\mathbb{C})$ is finitely complete because $\mathbb{C}$ is so. Lemma \ref{epiequiv} implies that regular epimorphisms in $\Equiv(\CC)$ are stable under pullbacks since regular epimorphisms are stable in $\mathbb{C}$, and regular epimorphisms in $\Equiv(\mathbb{C})$ are ``levelwise'' regular epimorphisms. The existence of the (regular epimorphism, monomorphism) factorisation of a morphism $(f,g)$ as in \eqref{internalfunctor} in the category $\Equiv(\mathbb{C})$ follows from the construction of diagram \eqref{factor}: the (regular epimorphism, monomorphism) factorisation $f=i\circ q$ of $f$ in $\CC$ gives rise to the (regular epimorphism, monomorphism) factorisation $g=j\circ \alpha$ of $g$ in $\CC$. Thus $(q,\alpha)\circ (i,j)$ is the (regular epimorphism, monomorphism) factorisation of $(f,g)$ in $\Equiv(\CC)$.
To see that $\Equiv(\mathbb{C})$ has the Goursat property one uses Theorem \ref{CKP}: to check that the regular image of an equivalence relation in the category $\Equiv(\mathbb{C})$ is again an equivalence in $\Equiv(\mathbb{C})$ one mainly uses the same (``levelwise'') property in the category $\mathbb C$.
\end{proof}

\section{Connectors and groupoids in Goursat categories}
In this section we prove that connectors are stable under quotients in any Goursat category $\CC$. We then define the category $\Conn(\CC)$ of connectors in $\CC$ whose objects are pairs of equivalence relations equipped with a connector, and prove that $\Conn(\CC)$ is a Goursat category whenever the base category $\CC$ is. We conclude by giving a new characterisation of Goursat categories in terms of properties of groupoids and internal categories.

\begin{defi}
Let $(R,r_1,r_2)$ and $(S, s_1,s_2)$ be two equivalence relations on an object $X$ and $R\times_X S$ the pullback of $r_2$ along $s_1$. A \textbf{connector} \cite{bmaconn} between $R$ and $S$  is an arrow $p: R \times_X S \longrightarrow  X$ in $\mathbb{C}$ such that
\begin{enumerate}
\item $x S p(x,y,z)R z$;
\item $p(x,x,y)= y$;
\item $p(x,y,y)= x$;
\item $p(x,y,p(z,u,v))= p(p(x,y,z),u,v)$,
\end{enumerate}
when each term is defined.
\end{defi}
\begin{remark}
Given two regular epimorphisms $d \colon X \twoheadrightarrow Y$ and $c \colon X \twoheadrightarrow Z$, a connector on the effective equivalence relations $\Eq(d)$ and $\Eq(c)$ is the same thing as an internal \emph{pregroupoid} in the sense of Kock \cite{Kock, Kock2} (see also the introduction of \cite{bma}, for instance, for a comparison between these two related notions and some additional references). In the context of Mal'tsev or Goursat categories connectors are useful to develop a centrality theory of non-effective equivalence relations.
\end{remark}

\begin{ex}
If $\nabla_X$ is the largest equivalence relation on an object $X$, then an associative Mal'tsev operation $$p: X\times X \times X \longrightarrow X$$ is precisely a connector between $\nabla_X$ and $\nabla_X$.
\end{ex}

Connectors provide a way to distinguish groupoids amongst reflexive graphs:
\begin{propC}[\cite{cpp}]\label{Example} Given a reflexive graph $$
      \xymatrix {
    X_1 \ar@<6pt>[r]^{d} \ar@<-6pt>[r]_{c} & X_0 \ar[l]|-{e}}
$$
in a finitely complete category $\mathcal C$
  (i.e. $d \circ e = 1_{X_0} = c\circ e$)  then the connectors between $\Eq(c)$ and $\Eq(d)$ are in bijections with the groupoid structures on this reflexive graph.
\end{propC}

It is well known that Goursat categories satisfy the so-called \emph{Shifting Property}~\cite{hg,bm}. In this context connectors are unique when they exist (Theorem 2.13 and Proposition 5.1 in~\cite{bm}):
accordingly, for a given pair of equivalence relations on the same object the fact of having a connector becomes a \emph{property}.

\begin{defi}
Let $R$ and $S$ be two equivalence relations on an object X. A \textbf{double equivalence relation} on $R$ and $S$ is given by an object $C \in \mathbb{C}$ equipped with two equivalence relations $(\pi_1,\pi_2): C \rightrightarrows R$ and $(p_1,p_2): C \rightrightarrows S$ such that the following diagram
\[
      \xymatrix {
   C  \ar@<.5ex>[r]^{\pi_1} \ar@<-.5ex>[r]_{\pi_2} \ar@<.5ex>[d]^{p_2} \ar@<-.5ex>[d]_{p_1}  & R \ar@<.5ex>[d]^{r_2} \ar@<-.5ex>[d]_{r_1} \\ S  \ar@<.5ex>[r]^{s_1} \ar@<-.5ex>[r]_{s_2} & X
}
\]
commutes (in the ``obvious'' way).
\end{defi}

A double equivalence relation $C$ on $R$ and $S$ is called a \textbf{centralizing relation} \cite{cpp} when the square
\[
      \xymatrix {
   C \ar@{}[dr]|(.25){\lrcorner}\ar[r]^{\pi_1} \ar[d]_{p_1} & R \ar[d]^{r_1} \\ S  \ar[r]_{s_1} & X
}
\]
is a pullback. Under this assumption it follows that any of the commutative squares in the definition of a centralizing relation is a pullback.

The following lemma gives the relationship between connectors and centralizing relations.
\begin{lemmaC}[\cite{bmaconn}] \label{concen}
If $\mathbb{C}$ is a category with finite limits and $R$ and $S$ are two equivalence relations on the same object $X$, then the following conditions are equivalent:
\begin{enumerate}
  \item[(i)] there exists a connector between $R$ and $S$;
  \item[(ii)] there exists a centralizing relation on $R$ and $S$.
\end{enumerate}
\end{lemmaC}

When $\CC$ is a Mal'tsev category, $R$ and $S$ are equivalence relations on an object $X$ with a connector  and $i:I\rightarrowtail X$ is a monomorphism, then the inverse images $i^{-1}(R)$ and $i^{-1}(S)$ also have a connector~\cite{bmaconn}.  We establish a similar property for Goursat categories, with respect to regular epimorphisms:

\begin{propo}\label{staquo}
Let $\mathbb{C}$ be a Goursat category, $R$ and $S$ two equivalence relations on an object $X$, and let $f: X \twoheadrightarrow Y$ be a regular epimorphism. If there exists a connector between $R$ and $S$, then there exists a connector between the regular images $f(R)$ and $f(S)$.
\end{propo}

\begin{proof} Suppose that there exists a connector between $R$ and $S$. This implies that there exists a centralizing relation $(C,(\pi_1,\pi_2),(p_1,p_2))$ on $R$ and $S$. Consider the regular image $(f(R), a, b)$ and $(f(S),c,d)$ of $R$ and $S$ along $f$. We obtain the following diagram
\begin{equation}
\label{cube}
\vcenter{\xymatrix {
     C \ar@{->>}[rr]^{\alpha} \ar@<.5ex>[rd]^{\pi_2} \ar@<-.5ex>[rd]_{\pi_1} \ar@<.5ex>[dd]^{p_2} \ar@<-.5ex>[dd]_{p_1}& {}  & f_R(C) \ar@<.5ex>@{-->}[dd]^(0.7){\alpha_2} \ar@<-.5ex>@{-->}[dd]_(0.7){\alpha_1} \ar@<.5ex>[rd]^{\beta_2} \ar@<-.5ex>[rd]_{\beta_1} & {} \\
     {} & R \ar@<.5ex>[dd]^(0.3){r_2} \ar@<-.5ex>[dd]_(0.3){r_1}  \ar@{->>}[rr]^(0.3){f_R} & {} & f(R) \ar@<.5ex>[dd]^{b} \ar@<-.5ex>[dd]_{a}\\
     S \ar@{-->>}[rr]^(0.7){f_S} \ar@<.5ex>[rd]^{s_2} \ar@<-.5ex>[rd]_{s_1}& {} & f(S) \ar@<.5ex>@{-->}[rd]^{d} \ar@<-.5ex>@{-->}[rd]_{c} & {} \\
     {} & X \ar@{->>}[rr]_{f}   & {} &  Y,
}}
\end{equation}
where $(f_R(C), \beta_1,\beta_2)$ is the regular image of the equivalence relation $(C,\pi_1,\pi_2)$ along the regular epimorphism $f_R$. The fact that the square
\[
      \xymatrix {
   {C\, } \ar@{->>}[r]^{\alpha} \ar[d]_{f_Sp_1} & f_R(C) \ar[d]^{\langle a\beta_1,a\beta_2\rangle} \ar@{.>}[ld]_-{\alpha_1}
   \\ {f(S)\,} \ar@{ >->}[r]_-{\langle c,d\rangle} & Y \times Y
}
\]
 commutes, $\alpha$ is a strong epimorphism and $\langle c,d\rangle$ is a
 monomorphism, implies the existence of an arrow $\alpha_1:
 f_R(C)\longrightarrow f(S)$ making the above diagram commute. Similarly, from
 the commutativity of the third diagram
\[
      \xymatrix {
   {C\, } \ar@{->>}[r]^{\alpha} \ar[d]_{f_Sp_2} & f_R(C) \ar[d]^{\langle b\beta_1,b\beta_2\rangle} \ar@{.>}[ld]_-{\alpha_2} \\
   {f(S) \,} \ar@{ >->}[r]_-{\langle c,d\rangle} & Y \times Y
}
\]
we obtain an arrow $\alpha_2: f_R(C)\longrightarrow f(S)$.

The relations $(f_R(C), \beta_1,\beta_2)$, $(f(R), a, b)$ and $(f(S),c,d)$ are all equivalence relations by Theorem~\ref{CKP}. It is then easy to check that the relation $(f_R(C),\alpha_1,\alpha_2)$ is an equivalence relation on $f(S)$. In fact, the morphism $\langle \alpha_1,\alpha_2\rangle \colon f_R(C)\to f(S)\times f(S)$ is a monomorphism since $\langle c\times c, d\times d\rangle \circ \langle \alpha_1,\alpha_2 \rangle= \langle a,b\rangle \times \langle a,b\rangle \circ \langle \beta_1, \beta_2\rangle$. So, $\langle \alpha_1,\alpha_2\rangle$ is the regular image of $\langle p_1,p_2\rangle$ along $f_S$, thus it is an equivalence relation on $f(S)$ by Theorem~\ref{CKP}.

By assumption all the left squares of \eqref{cube} are pullbacks, so it follows that all the right squares of \eqref{cube} are pullbacks as well by Theorem \ref{newcharact} (ii). Then $(f_R(C), (\alpha_1,\alpha_2),(\beta_1,\beta_2))$ is a centralizing relation on $f(R)$ and $f(S)$. By Lemma \ref{concen} there is a connector between  $f(R)$ and $f(S)$.
\end{proof}
\vspace{1cm}

We are now going to show that the category whose objects are pairs of equivalence relations equipped with a connector is a Goursat category whenever the base category is a Goursat category. For this, let us first fix some notation: if $\mathbb{C}$ is a finitely complete category, we write 2-$\Eq(\mathbb{C})$ for the category whose objects $(R,S,X)$ are pairs of equivalence relations $R$ and $S$ on the same object $X$
\[
      \xymatrix {
    R \ar@<.5ex>[r]^{r_1} \ar@<-.5ex>[r]_{r_2} & X & S \ar@<.5ex>[l]^{s_2} \ar@<-.5ex>[l]_{s_1}
}
\]
and arrows are triples $(f_R,f_S,f)$ making the following diagram commute:
\begin{equation}
\label{arrows in Eq(C)}
    \vcenter{\xymatrix {
    R \ar@<.5ex>[r]^{r_1} \ar@<-.5ex>[r]_{r_2} \ar[d]_{f_R} & X  \ar[d]_f & S \ar@<.5ex>[l]^{s_2} \ar@<-.5ex>[l]_{s_1} \ar[d]^{f_S} \\ \bar{R} \ar@<.5ex>[r]^{\bar{r_1}} \ar@<-.5ex>[r]_{\bar{r_2}} & \bar{X} & \bar{S}. \ar@<.5ex>[l]^{\bar{s_2}} \ar@<-.5ex>[l]_{\bar{s_1}}
}}
\end{equation}
We write $\Conn(\mathbb{C})$ for the category whose objects $(R,S,X,p)$ are pairs of equivalence relations $R$ and $S$ on an object $X$ with a given connector $p: R \times_X S \rightarrow X$; arrows in $\Conn(\mathbb{C})$ are arrows in 2-$\Eq(\mathbb{C})$ respecting the connectors. This means that, given a diagram \eqref{arrows in Eq(C)} where both $(R,S,X)$ and $(\bar{R},\bar{S},\bar{X})$ are in $\Conn(\mathbb{C})$, with $p: R \times_X S \rightarrow X$ and $\bar{p}: \bar{R} \times_Y \bar{S} \rightarrow Y$ the corresponding connectors,  then the diagram
\[
      \xymatrix {
     R \times_X S \ar[r]^{\bar{f}} \ar[d]_{p}   & \bar{R} \times_{\bar{X}} \bar{S} \ar[d]^{\bar{p}} \\ X  \ar[r]_{f} & \bar{X}
}
\]
commutes, where $\bar{f}$ is the natural map induced by the universal property of the pullback $\bar{R} \times_{\bar{X}} \bar{S}$.

We say that a subcategory $\mathbb{P}$ is \emph{closed under (regular) quotients} in a category $\mathbb{Q}$ if, for any regular epimorphism $f: A \twoheadrightarrow B$ in $\mathbb{Q}$ such that $A \in \mathbb{P}$, then $B \in \mathbb{P}$.

\begin{propo}\label{stabilityepi}
\label{connclo}
If $\mathbb{C}$ is a Goursat category, then $\Conn(\mathbb{C})$ is a full subcategory of  2-$\Eq(\mathbb{C})$, that is closed in 2-$\Eq(\mathbb{C})$ under quotients.
\end{propo}
\begin{proof}
The fullness of the forgetful functor $\Conn(\mathbb{C}) \rightarrow$ 2-$\Eq(\mathbb{C})$ follows from Corollary 5.2 in~\cite{bm}, by taking into account the fact that any Goursat category satisfies the Shifting Property.

Let us then consider a regular epimorphism in 2-$\Eq(\mathbb{C})$
\[
      \xymatrix {
    R \ar@<.5ex>[r]^{r_1} \ar@<-.5ex>[r]_{r_2} \ar@{>>}[d]_{f_R} & X  \ar@{>>}[d]_f & S \ar@<.5ex>[l]^{s_2} \ar@<-.5ex>[l]_{s_1} \ar@{>>}[d]^{f_S} \\ \bar{R} \ar@<.5ex>[r]^{\bar{r_1}} \ar@<-.5ex>[r]_{\bar{r_2}} & \bar{X} & \bar{S} \ar@<.5ex>[l]^{\bar{s_2}} \ar@<-.5ex>[l]_{\bar{s_1}}
}
\]
 (this means that $f$, $f_R$ and $f_S$ are regular epimorphisms in $\mathbb{C}$) such that its domain $(R,S,X)$ belongs to $\Conn(\mathbb{C})$. The equalities $f(R)= \bar{R}$ and $f(S)= \bar{S}$, together with Proposition \ref{staquo}, imply that there exists a connector between $\bar{R}$ and  $\bar{S}$.
\end{proof}

\begin{lemma}\label{fullsub}
Let  $\mathbb{D}$ be a finitely complete category, and $\mathbb{C}$ a full subcategory of  $\mathbb{D}$ closed in $\mathbb{D}$ under finite limits and quotients. Then:
\begin{enumerate}
\item $\mathbb{C}$ is regular whenever $\mathbb{D}$ is regular.
\item $\mathbb{D}$ is a Goursat category whenever $\mathbb{C}$ is a Goursat category.
\end{enumerate}
\end{lemma}
\begin{proof} The (regular epimorphism, monomorphism) factorisation in  $\mathbb{D}$ of an arrow in  $\mathbb{C}$ is also its factorisation in  $\mathbb{C}$, since $\mathbb{C}$ is closed in $\mathbb{D}$ under quotients. Since finite limits in $\mathbb{C}$ are calculated as in $\mathbb{D}$, it follows that regular epimorphisms are stable under pullbacks. Now the second statement easily follows from the fact that the composition of relations is computed in the same way in $\mathbb{C}$ and in $\mathbb{D}$. \end{proof}

\begin{theorem}
\label{conngour}
 If $\mathbb{C}$ is a Goursat category then $\Conn(\mathbb{C})$ is a Goursat category.
 \end{theorem}
\begin{proof}
  Using similar arguments as those given in the proof of Proposition \ref{equiGour} with respect to $\Equiv(\CC)$, one may deduce that 2-$\Eq( \mathbb{C})$ is a Goursat category. The result then follows from Proposition \ref{stabilityepi} and Lemma \ref{fullsub}.
\end{proof}

We finally prove that internal categories and groupoids can be used to characterise Goursat categories. Recall that an \textbf{internal category} in a category $\mathbb{C}$ with pullbacks is a reflexive graph with a multiplication $m \colon X_1 \times_{X_0}X_1 \rightarrow X_1$
$$
      \xymatrix {
    X_1 \times_{X_0}X_1   \ar@<6pt>[r]^-{p_1} \ar@<-6pt>[r]_-{p_2} \ar[r]|-{m}    & X_1 \ar@<6pt>[r]^{d} \ar@<-6pt>[r]_{c} & X_0, \ar[l]|-{e}
}
$$
(where $X_1 \times_{X_0}X_1$ is the pullback of $d$ and $c$) such that:
\begin{itemize}
\item $d\circ m=d \circ p_2$,~~ $c \circ m=c\circ p_1$,~~ $m \circ \langle e\circ d, 1_{X_1}\rangle = 1_{X_1}= m \circ \langle 1_{X_1},e\circ c\rangle$;
\item $ m\circ (1\times m) = m\circ (m \times 1)$.
\end{itemize}
An internal category
$$
      \xymatrix {
    X_1 \times_{X_0}X_1   \ar@<6pt>[r]^-{p_1} \ar@<-6pt>[r]_-{p_2} \ar[r]|-{m}    & X_1  \ar@<6pt>[r]^{d} \ar@<-6pt>[r]_{c} & X_0, \ar[l]|-{e}
}
$$
is a \textbf{groupoid} when there is an additional morphism $ i \colon  X_1 \longrightarrow X_1  $
satisfying the axioms:
\begin{itemize}
\item $d\circ i = c$, ~~$c\circ i=d$;
\item $m\circ \langle i, 1_{X_1}\rangle= e \circ c$ ~~and ~~$m\circ \langle 1_{X_1} , i\rangle= e\circ d $.
\end{itemize}
We write $\Cat(\mathbb{C})$ for the category of internal categories in  $\mathbb{C}$ (and internal functors as morphisms),  $\Grpd(\mathbb{C})$ for the category of groupoids in $\mathbb{C}$,  and $\RG(\mathbb{C})$ for the category of reflexive graphs in $\mathbb{C}$ (with obvious morphisms).

An equivalence relation is a special kind of groupoid, where its domain and codomain morphisms are jointly monomorphic; also any reflexive and transitive relation is in particular an internal category. If $\CC$ is a Goursat category, then any reflexive and transitive relation is an equivalence relation or, equivalently, any internal category is a groupoid (Theorem 1 in~\cite{ndt}). Then Theorem ~\ref{CKP}, which could equivalently be stated through the property that $\Equiv(\CC)$ (or the category of reflexive and transitive relations in $\CC$) is closed in the category of reflexive relations in $\CC$ under quotients, has an \emph{extended} counterpart given below. This characterisation leads to the observation that the structural aspects of Goursat categories mainly concern groupoids (rather than equivalence relations).

\begin{theorem}
\label{caract}
Let $\mathbb{C}$ be a regular category. Then the following conditions are equivalent:
\begin{enumerate}
  \item[(i)] $\mathbb{C}$ is a Goursat category;
  \item[(ii)] $\Grpd(\mathbb{C})$ is closed in $\RG(\mathbb{C})$ under quotients;
  \item[(iii)] $\Cat(\mathbb{C})$ is closed in $\RG(\mathbb{C})$ under quotients.
\end{enumerate}
\end{theorem}
\begin{proof}
\leavevmode
\begin{description}[beginpenalty=99,before={\renewcommand\makelabel[1]{##1}}]
  \item[(i) $\Rightarrow$ (ii)] Let
\[
      \xymatrix {
     X_1\ar@{>>}[r]^{g} \ar@<6pt>[d]^{c} \ar@<-6pt>[d]_{d}
     & {X'_1} \ar@<6pt>[d]^{c'} \ar@<-6pt>[d]_{d'} \\ X_0 \ar[u]|-{e} \ar@{>>}[r]_{f} & X'_0 \ar[u]|-{e'}
}
\]
be a regular epimorphism $(f,g)$ in $\RG(\mathbb{C})$ (which means that $f$ and $g$ are regular epimorphisms in $\CC$), with
$$
      \xymatrix {
    X_1 \ar@<6pt>[r]^{d} \ar@<-6pt>[r]_{c} & X_0 \ar[l]|-{e}
}
$$
a groupoid in $\mathbb{C}$. By Proposition~\ref{Example}, there exists a connector between $\Eq(d)$ and $\Eq(c)$. Let $\Eq(d)$, $\Eq(c)$, $\Eq(d')$ and $\Eq(c')$ be the kernel pairs of the arrows $d$, $c$, $d'$ and $c'$, respectively. Let  $\lambda : \Eq(d)\rightarrow \Eq(d')$ and $\beta : \Eq(c)\rightarrow \Eq(c')$ be the arrows induced by the universal property of kernel pairs $\Eq(d')$ and $\Eq(c')$, respectively. By Theorem \ref{Goursatpushout}, $\lambda$ and $\beta$ are regular epimorphisms, so that $g(\Eq(d)) = \Eq(d')$ and $g(\Eq(c)) = \Eq(c')$. By Proposition \ref{staquo} there is then a connector between $\Eq(d')$ and $\Eq(c')$, thus
$$
      \xymatrix {
    X'_1 \ar@<6pt>[r]^{d'} \ar@<-6pt>[r]_{c'} & X'_0 \ar[l]|-{e'}
}
$$
is a groupoid (Proposition~\ref{Example}).

\item[(ii) $\Rightarrow$ (i)] This implication follows from Theorem~\ref{CKP} and the fact that equivalence relations are in particular groupoids (whose domain and codomain morphisms are jointly monomorphic).

\item[(i) $\Rightarrow$ (iii)] This implication follows from (i) $\Rightarrow$ (ii) and the fact that $\Grpd({\mathbb{C}}) \cong \Cat(\mathbb{C})$ in a Goursat context, as recalled above.

\item[(iii) $\Rightarrow$ (i)] Let $(R, r_1, r_2)$ be an equivalence relation on $X$, $f: X \twoheadrightarrow Y$ a regular epimorphism and $(f(R), t_1, t_2)$ the regular image of $R$ along $f$
$$
\label{image}
    \vcenter{\xymatrix {
     R \ar@{->>}[r]^{g} \ar@<.5ex>[d]^{r_2} \ar@<-.5ex>[d]_{r_1}  & f(R) \ar@<.5ex>[d]^{t_2} \ar@<-.5ex>[d]_{t_1} \\ X  \ar@{->>}[r]^{f} & Y.
}}
$$
$(f(R), t_1, t_2)$ is reflexive and symmetric being the image of the equivalence relation $R$ along a regular epimorphism $f$. By assumption, $(f(R), t_1, t_2)$ is an internal category, thus it is an equivalence relation. It follows that $\mathbb{C}$ is a Goursat category (by Theorem \ref{CKP}).
\qedhere
\end{description}

\end{proof}

\begin{remark}
Observe that Theorem \ref{caract} also implies that $\Grpd(\mathbb{C})$ and $\Cat(\mathbb{C})$ are Goursat categories whenever $\mathbb C$ is, again thanks to Lemma \ref{fullsub}, the category $\RG(\mathbb{C})$ obviously being a Goursat category. This simplifies and slightly extends Proposition $4.3$ in \cite{gro}, where the existence of coequalizers in $\mathbb C$ was assumed.
\end{remark}
\begin{remark}
A result analogous to Theorem \ref{caract} holds in the Mal'tsev context: a category $\mathbb{C}$ is a Mal'tsev category if and only if $\Grpd(\mathbb{C})$ (or, equivalently, $\Cat(\CC)$) is closed in $\RG(\mathbb{C})$ under subobjects~\cite{bournfib}. Together with the comments made before Proposition~\ref{staquo} we observe the existence of a sort of duality between Mal'tsev categories and Goursat categories. Similar results hold for Mal'tsev categories with respect to monomorphisms and for Goursat categories with respect to regular epimorphisms.
\end{remark}

\end{document}